\documentclass[letterpaper, 10 pt, journal, twoside]{ieeetran}  

\IEEEoverridecommandlockouts                              

\usepackage{graphics} 
\usepackage{amssymb}  
\usepackage{cite}
\usepackage{mathtools}
\usepackage[usenames,dvipsnames]{color}
\usepackage{soul}
\usepackage{mathrsfs}
\usepackage{bm}
\usepackage{color}
\usepackage{verbatim} 

\allowdisplaybreaks[3]

\newtheorem{thm}{Theorem}[section]
\newtheorem{lem}[thm]{Lemma}

\newtheorem{cor}[thm]{Corollary}
\newtheorem{dfn}[thm]{Definition}
\newtheorem{eg}[thm]{Example}

\usepackage{graphicx,color,amsmath}   

\newcommand{\mc}[1]{\mathcal{#1}}

\newcommand{\norm}[2]{\left\| #1 \right\|_#2 }
\newcommand{\normm}[1]{\left\| #1 \right\| }

\newcommand{\R}{\mathbb{R}}

\newcommand{\state}{\bm{x}}
\newcommand{\inp}{\bm{u}}
\renewcommand{\ss}{{\mathbb{X}}}  
\newcommand{{\cs}}{{\mathbb{U}}}
\newcommand{{\as}}{{\mathbb{K}}}
\newcommand{\A}{{\mathcal{A}}}
\newcommand{\T}{{\mathcal{T}}}
\newcommand{\F}{{\mathcal{F}}}
\newcommand{\B}{{\mathcal{B}}}
\newcommand{\K}{{\mathscr{K}}}
\newcommand{\EL}{{\mathscr{L}}}
\newcommand{\KL}{{\mathscr{KL}}}

\title{Stability and Well-posedness  of a Nonlinear Railway Track Model}

\author{M. Sajjad Edalatzadeh and Kirsten A. Morris
\thanks{M. S. Edalatzadeh is with the Department of Applied Mathematics, University of Waterloo, Waterloo, Ontario, Canada
        {\tt\small msedalat@uwaterloo.ca}}%
\thanks{K. A. Morris is with Faculty of Mathematics, Applied Mathematics, University of Waterloo, Waterloo, Ontario, Canada
        {\tt\small kmorris@uwaterloo.ca}}%
}

\begin{document}

\maketitle
\thispagestyle{empty}
\pagestyle{empty}

\begin{abstract}
Railway tracks rest on a foundation known for exhibiting nonlinear viscoelastic behavior. Railway track deflections are modeled by a semilinear partial differential equation. This paper studies the stability of solutions to this equation in presence of an input. With the aid of a suitable Lyapunov function, existence and exponential stability of classical solutions is established for certain inputs. The Lyapunov function is further used to find an a-priori estimate of the solutions, and also to study the input-to-state stability (ISS) of mild solutions.
\end{abstract}


\section{INTRODUCTION}
\IEEEPARstart{S}{tability} analysis of nonlinear partial differential equations (PDE's) modeling flexible structures has attracted attention in the past few decades.
To name but a few of the publications in this field; in \cite{fu2001boundary},  boundary stabilization of a nonlinear beam clamped at one end and supported by a nonlinear bearing at the other end is studied.
In \cite{favsangova2001asymptotic}, authors investigated asymptotic behavior of a semilinear viscoelastic beam model including a memory term. The nonlinearity is assumed to satisfy some growth assumption.
In \cite{takeda2012initial}, asymptotic stability of Falk model of shape memory alloys is studied using energy method. 
In \cite{edalatzadeh2014suppression,edalatzadeh2016boundary}, boundary stabilization of a nonlinear micro-beam model is studied using linearizion technique together with Lyapunov method.
The von Karman model of slender beams has also been investigated in many papers. In a recent study, Liu et al. considered asymptotic stability of von Karman beam with thermo-viscoelastic damping \cite{liu2018asymptotic}.
Nonlinear PDE's with a fourth order spatial derivative are not limited to flexible structures; an example is Cahn-Hilliard equation with inertial term which describes the phase separation of binary fluids; see for example, \cite{grasselli20092d}. 

The nonlinear railway track model in this paper was described in \cite{ansari2011}. 
The nonlinearity in this model is caused by the railway  support  ballast which is known for highly nonlinear viscoelastic behavior \cite{dahlberg2002dynamic}. Another nonlinear model for railway tracks has been suggested \cite{ding2013dynamic}. Unlike the previous model, this model also includes shear deformations in flexible track beams. This railway track model  was used to study the effect of passing vehicles on pavements \cite{ding2012convergence}. There are also studies devoted to the vibration monitoring and control of railway tracks  \cite{kouroussis2015review,zhu2015low}. Track deflections induced by train passage are a cause of ride discomfort, fatigue in railway, and disturbances to nearby buildings \cite{kouroussis2015review}. Thus, the vibrations need to be carefully monitored and controlled.

This paper focuses on well-posedness and stability, with respect to both initial conditions, and inputs of this model.
Input-to-state stability  (ISS) does not generalize in a straightforward way to infinite dimensions; see \cite{mironchenko2016local} for counter examples.  Input-to-state stability theory has been extended in recent years to include systems of infinite dimension.  In \cite{mironchenko2017}, a comparison between ISS theory of finite-dimensional systems and that of infinite-dimensional systems is presented.  
In \cite{Jayawardhana2008}, ISS of a class of linear infinite dimensional control systems with nonlinear feedbacks is discussed. The nonlinear feedback satisfies a sector condition that does not apply to the nonlinearity in this paper. In \cite{mazenc2011strict}, strict Lyapunov functions were used to investigate the stability of nonlinear heat equation; also, such Lyapunov functions were used to establish a robust stability of the equation in presence of a convection term and uncertainties. Integral input-to-state stability  is a weaker property than ISS. In \cite{mironchenko2016integral}, integral ISS  is discussed. In \cite{jacob2018infinite}, the relation between iISS and ISS is discussed for linear systems with an unbounded control operator. ISS with respect to boundary inputs and disturbances has also been studied in \cite{karafyllis2016iss,orlov2017general,pisano2017iss}. 
In \cite{pisano2017iss}, ISS of the reaction-diffusion-advection equation with boundary and in-domain point-wise sensing and actuation is considered. 

In this paper, a Lyapunov function is used to establish ISS for the nonlinear controlled railway track model. To construct the Lyapunov function, the multiplier method \cite{komornik1994exact} is used. 
Furthermore, a density argument is used to prove the ISS of the model when the inputs are not differentiable. In such cases, the Lyapunov function is also non-differentiable. A density argument was also used in \cite[Lemma~2.2.3]{mironchenThesis}, where the control system is assumed to have a transition map that continuously depends on both initial conditions and inputs.

The paper is organized as follows: Section 2 is a short section containing notation and definitions. Section 3 introduces the railway track model and discusses well-posedness of this model. Section 4 presents   existence and exponential stability of classical solutions to the PDE  while section 5 is devoted to the stability of mild solutions. 

\section{Notation and definitions}
Let $\ss$ and $\cs$ be Banach spaces, and $I\subset \R$ be a possibly unbounded interval. The function space $C^m(I;\ss)$ consists of all $m$ times continuously differentiable $\ss$-valued functions, $L^2_{loc}(I;\cs)$ is the space of all strongly measurable functions $\inp:I\to \cs$, $t\mapsto \inp(t)$, for which $\norm{\inp (t)}{\cs}$ is in $L^2_{loc}(I,\R)$. Denoted by $PC(\R^+;\cs)$ is the space of all bounded, piecewise continuous $\cs$-valued function over $\R^+$. 

The comparison function sets $\K$, $\K_\infty$, $\EL$, and $\KL$ are defined as
\begin{flalign}
\allowdisplaybreaks \K\coloneqq &\{\gamma:\R^+ \to\R^+ |\, \gamma \text{ is continuous,}\notag\\ 
\allowdisplaybreaks&\text{ strictly increasing, and } \gamma (0) = 0\}, \\
\K_\infty\coloneqq& \{\gamma\in \K|\, \gamma \text{ is unbounded}\},\\
\EL \coloneqq &\{\gamma:\R^+ \to\R^+ |\, \gamma \text{ is continuous,}\notag\\ 
&\text{ strictly decreasing, and } \lim\limits_{t\to \infty}\gamma (t) = 0\}, \\
\KL \coloneqq &\{\beta: \R^+ \times \R^+ \to\R^+ |\, \beta \text{ is continuous,}\notag \\
&\;\beta(\cdot, t)\in \K, \beta(r, \cdot)\in \EL,\, \forall  t\ge 0,\; \forall r>0\}.
\end{flalign}

Let $\state$ indicate the state, and $\inp$  the input. For  linear operators $\A:D(\A)\subset \ss \to \ss$,   $\B:\cs\to\ss$,  and possibly nonlinear operator $\F(\cdot):\ss\to \ss$, consider the initial value problem
\begin{equation}
\left\{\begin{array}{ll}
\dot{\state}(t)=\mc{A}\state(t)+\mc{F}(\state(t))+\mc{B}\inp (t), \; t>0,\\ \state(0)=\state_0\in \ss \label{eq: IVP} \tag{IVP}.
\end{array}\right.
\end{equation} 

\begin{dfn}\cite[Definition~4.2.1]{pazy} (Classical Solution) \label{def:classical}
A function $\state:[0,T)\to \ss$ is a classical solution to (\ref{eq: IVP}) on $[0,T)$ if $\state$ is continuous on $[0,T)$, continuously differentiable on $(0,T)$, $\state(t)\in D(\A)$ for $0<t<T$, and (\ref{eq: IVP}) is satisfied.
\end{dfn}

\begin{dfn}\label{defn-mild} (Mild Solution)
Let $\A$ be the infinitesimal generator of a strongly continuous semigroup $\T(t)$. If $\state\in C([0,T];\ss)$ satisfies
\begin{flalign}\
\state(t)=&\mc{T}(t)\state_0\label{eq:mild solution}\\
&+\int_0^t \mc{T}(t-s)\mc{F}(\state(s))ds+\int_0^t \mc{T}(t-s)\mc{B}\inp (s)ds, \notag
\end{flalign} 
for every $\state_0\in \ss$,
it is said to  be a {\em mild solution} to (\ref{eq: IVP}).
\end{dfn}

In the following definitions, it is assumed that a unique mild solution to (\ref{eq: IVP}) exists for any $\inp \in PC(\R^+;\cs)$.
\begin{dfn}
\cite[Definition~9]{mironchenko2017} (Input-to-state Stability)\label{def:ISS stability}
The (\ref{eq: IVP}) is called input-to-state stable (ISS) if there exist $\beta\in \KL$ and $\gamma\in \K$ such that for all $\state_0\in \ss$, $\inp \in PC(\R^+;\cs)$, and $t>0$ the mild solution (\ref{eq:mild solution}) satisfies
\begin{equation}\label{ISS inequality}
\normm{\state(t)}\le \beta(\normm{\state_0},t)+\gamma\left(\sup_{t \geq 0}\norm{\inp(t)}{\cs}\right).
\end{equation}
\end{dfn}

The Dini derivative of a function $V:D(\subset \ss)\to\R^+$ along trajectories of (\ref{eq: IVP}) is 
\begin{equation}
\dot{V}_{\inp}(\state_0)\coloneqq \limsup_{t\to0^+}\frac{1}{t} (V(\state(t))-V(\state_0)).
\end{equation}
\begin{dfn} \cite[Definition~11]{mironchenko2017} (ISS Lyapunov Function) \label{def:ISS Lyapunov}
A continuous function $V:D(\subset \ss)\to\R^+$ is called an ISS Lyapunov function on $D$, if there exist $\psi_1, \, \psi_2\in \K_\infty$, $\alpha\in \K_\infty$, and $\sigma \in \K$ such that for all $\state_0\in\ss$, $\inp(t)\in PC(\R^+;\cs)$,
\begin{equation}
\psi_1(\normm{\state})\le V(\state)\le \psi_2 (\normm{\state}), \; \forall \state\in\ss,
\end{equation}
and 
\begin{equation}
\dot{V}_{\inp}(\state_0)\le -\alpha(\normm{\state_0})+\sigma\left(\sup_{t \geq 0}\norm{\inp(t)}{\cs}\right).
\end{equation}
\end{dfn}
\section{WELL-POSEDNESS of  RAILWAY TRACK MODEL}
Railway tracks rest on ballast which is known to exhibit nonlinear viscoelastic behavior. Considering the Kelvin-Voigt damping in the track beam, the semi-linear partial differential equation governing the motion of the track $w(\xi,t)$ on $\xi\in [0,\ell]$ is \cite{ansari2011}
\begin{equation}\label{eq:rail-IVP}
\begin{cases}
\rho a \frac{\partial^2 w}{\partial t^2}+\frac{\partial}{\partial \xi^2}(EI\frac{\partial^2 w}{\partial \xi^2}+C_d\frac{\partial ^3 w}{\partial \xi^2 \partial t})+\mu \frac{\partial w}{\partial t}+kw\\
\allowdisplaybreaks\quad +\alpha w^3=u(\xi,t),\\
\allowdisplaybreaks w(\xi,0)=w_0(\xi), \quad \frac{\partial w}{\partial t}(\xi,0)=v_0(\xi),\\
\allowdisplaybreaks w(0,t)=w(\ell,t)=0,\\
\allowdisplaybreaks EI\frac{\partial ^2w}{\partial \xi^2}(0,t)+C_d\frac{\partial ^3w}{\partial \xi^2\partial t}(0,t)=0,\\[1mm]
EI \frac{\partial^2 w}{\partial \xi^2}(\ell,t)+C_d\frac{\partial ^3w}{\partial \xi^2\partial t}(\ell,t)=0, 
\end{cases}
\end{equation}
where the positive constants $E$, $I$, $\rho$, $a$, and $\ell$ are the modulus of elasticity, second moment of inertia, density of the beam, cross-sectional area, and length of the beam, respectively. The linear and nonlinear parts of the foundation elasticity correspond to the positive coefficients $k$ and $\alpha$, respectively. The constant $\mu>0$ is the damping coefficient of the foundation, and $C_d\ge 0$ is the coefficient of Kelvin-Voigt damping in the beam.   
The external force exerted on the railway track by moving trains, active dampers, or other external force,  is denoted by $u(\xi,t).$
The model considered here differs from that in  \cite{ansari2011} by the inclusion of   Kelvin-Voigt damping in the beam  if $C_d$ has a non-zero value, although $C_d>0$ is not assumed in the analysis.

Let $v=\partial w/\partial t$. Define the state space $\ss=H^2(0,\ell )\cap H_0^1(0,\ell )\times L^2(0,\ell )$  with  norm
\begin{equation}
\| (w,v) \|^2=\int_0^{\ell} EIw_{\xi \xi}^2+kw^2+\rho a v^2 \, d\xi \label{eq: norm},
\end{equation}
where the subscript $\cdot_\xi$ denotes the derivative with respect to $\xi$. Define the closed self-adjoint positive operator 
\begin{flalign}
&\mc{A}_0w\coloneqq w_{\xi \xi \xi \xi},\notag \\
&D(\mc{A}_0)\coloneqq\left\lbrace w\in H^4(0,\ell )| \, w(0)=w(\ell)=0,\right.\notag\\
&\qquad \qquad \left. w_{\xi \xi}(0)=w_{\xi \xi}(\ell)=0 \right\rbrace,
\end{flalign}
and also define
\begin{equation}
\mc{A}_{\scriptscriptstyle KV}(w,v)\coloneqq\left(v,-\frac{1}{\rho a}\mc{A}_0(EIw+C_dv)\right),
\end{equation}
with 
\begin{flalign}
D(\mc{A}_{\scriptscriptstyle KV})\coloneqq&\left\lbrace(w,v)\in \ss| \, v\in H^2(0,\ell )\cap H_0^1(0,\ell ), \right.\notag \\
& \left. EIw+C_dv\in D(\mc{A}_0) \right\rbrace.
\end{flalign}
Also, let $\inp \in \cs\coloneqq L^2(0,\ell)$, and define the linear operators $\mc{K}$, $\A$, $\mc{B}$, and the nonlinear operator $\mc{F}(\cdot)$ as
\begin{flalign}
&\allowdisplaybreaks\mc{K}(w,v)\coloneqq(0,-\frac{1}{\rho a}(\mu v + kw)),\\
&\allowdisplaybreaks\mc A  \coloneqq\mc{A}_{\scriptscriptstyle KV}+\mc{K}, \text{ with } D (\mc A  ) = D(\mc{A}_{\scriptscriptstyle KV} ),\\
&\allowdisplaybreaks\mc{B}\inp\coloneqq(0,\frac{1}{\rho a}\inp),\\
&\allowdisplaybreaks\mc{F}(w,v)\coloneqq(0,-\frac{\alpha}{\rho a} w^3).\label{eq:cubic nonlinearity}
\end{flalign}
With these definitions and by setting the state $\state(t)=(w(\cdot,t),v(\cdot,t))$, initial condition $\state_0=(w_0(\cdot),v_0(\cdot))$, and the input $\inp (t)=u(\cdot,t)$, the state space representation of the railway track IVP is
\begin{equation}
\left\{\begin{array}{ll}
\dot{\state}(t)=\mc{A}\state(t)+\mc{F}(\state(t))+\mc{B}\inp (t), \; t>0,\\ \state(0)=\state_0\in \ss\notag.
\end{array}\right.
\end{equation}
Notice that the nonlinear term $w^3$ is in $L^2(0,\ell)$ since $H^2(0,\ell) \subset C([0,\ell])$. Thus, the nonlinear operator $\mc{F}(\cdot)$ is well-defined on $\ss$. It is also locally Lipschitz continuous; see \cite[Lem. 6.1]{edalatzadehSICON} where it is shown to be continuously Fr\'echet differentiable on $\ss$. Also, the bounded operator $\mc{B}$ maps an input $\inp \in L^2(0,\ell)$ into the state space $\ss$. This  input space is used in many applications.

It is well known that $\mc A_{\scriptscriptstyle KV}$ generates a strongly continuous contraction semigroup on $\ss$; see \cite{chen1989proof}.
The operator $\mc{K}$ is a bounded linear operator on $\ss$ and so
the operator $\mc{A}$, with the same domain as $\mc{A}_{\scriptscriptstyle KV}$,
generates a strongly continuous semigroup on  $ \ss $ \cite[Cor. 3.2.2]{pazy}. The assumption that $\mu>0$ implies that $\mc A$ generates an exponentially stable semigroup.

\begin{thm}\cite[Theorem~3.1]{edalatzadehSICON}\label{thm:existence mild}
Let $\A$ be the infinitesimal generator of a strongly continuous semigroup. If the nonlinear operator $\F(\cdot)$ is locally Lipschitz continuous on $\ss$, then for every $\state_0\in \ss$ and positive number $R$, there exist $T>0$ such that (\ref{eq: IVP}) admits a unique mild solution $\state\in C([0,T];\ss)$ for all $\inp \in L^p(0,T;\cs)$, $\norm{\inp}{{L^p(0,T;\cs)}}\le R$, $1<p<\infty$.
\end{thm}
Thus, by Theorem \ref{thm:existence mild}, a unique local (in time) mild solution to railway IVP is ensured.  

If the input admits further regularity, a mild solution is also a classical solution.

\begin{thm}\cite[Theorem~6.1.5]{pazy}\label{thm:existence classical}
Let $\A$ be the infinitesimal generator of a strongly continuous semigroup $\T(t)$ on $\ss$. If $\inp \in C^1([0,T];\cs)$ and the nonlinear operator $\F(\cdot)$ is continuously Fr\'echet differentiable on $\ss$, then the mild solution of (\ref{eq: IVP}) with $\state_0\in D(\A)$ is a classical solution.
\end{thm}
 
\section{STABILITY of CLASSICAL SOLUTIONS}
The existing literature does not predict the existence of a global solution to the railway track PDE model. To investigate the existence and stability of a global solution, a Lyapunov method \cite{queirozbook} together with the multiplier method \cite{komornik1994exact} is used. Let $c$ be a constant to be determined. Define, for  any $\state\in \ss$, 
\begin{flalign}
V(\state)&=\label{eq:lyapunov function}\\
&\int_0^{\ell} EI (w_{\xi \xi})^2+kw^2+\frac{\alpha}{2} w^4+\rho a v^2+2c w v \, d\xi.\notag 
\end{flalign}
\begin{lem}\label{lem:bound}
Let $c$ satisfy 
\begin{equation}\label{eq:condition}
0<c<\sqrt{\rho k a}.
\end{equation} 
Then, there exist positive numbers $c_l$, $c_u$, and $c_e$ such that for every $\state\in \ss$, 
\begin{equation}
c_l\normm{\state}^2\le V(\state)\le c_u\normm{\state}^2+c_h\normm{\state}^4.
\label{eq:Vbound}
\end{equation}
\end{lem}

\begin{proof}
Young's inequality implies that for all $\epsilon_1>0$
\begin{flalign}
\Big| \int_0^{\ell} 2cwv \, d\xi\Big| \le c\int_0^\ell \epsilon_1 w^2+\frac{1}{\epsilon_1}v^2 \, d\xi.\label{eq:young}
\end{flalign}
This inequality gives the following lower bound on $V$:
\begin{flalign}
V(\state)&\ge\int_0^{\ell} EI (w_{\xi \xi})^2+\left( k-c\epsilon_1\right) w^2+\frac{\alpha}{2} w^4\notag\\
&\qquad +\left(\rho a-\frac{c}{\epsilon_1}\right) v^2 \, d\xi. \label{eq:stronger}
\end{flalign}
Define 
\begin{equation}
c_l=\min\{1-\frac{c\epsilon_1}{k},1-\frac{c}{\rho a\epsilon_1}\}.
\end{equation}
The condition (\ref{eq:condition}) on $c$ implies that there exists a number $\epsilon_1$ satisfying $c/{\rho a}<\epsilon_1<k/c$, which ensures that $c_l >0$.
The inequality (\ref{eq:stronger}) can then be re-written as 
\begin{equation}
V(\state)\ge c_l\normm{\state}^2. \label{eq:d1}
\end{equation} 

Furthermore, apply the inequality (\ref{eq:young}) to (\ref{eq:lyapunov function}) and define $c_u=1+\max\{c\epsilon_1/k,c/(\epsilon_1\rho a)\}$ to obtain
\begin{flalign}
V(\state)&\le \int_0^{\ell} EIw_{\xi \xi}^2+(k+c\epsilon_1)w^2+\frac{\alpha}{2}w^4\notag\\
&\qquad+(\rho a+\frac{c}{\epsilon_1})v^2 \, d\xi \notag\\
&\le c_u\normm{\state}^2+\frac{\alpha}{2}\int_0^{\ell} w^4 \, d\xi.
\end{flalign}
Recall  the continuous embedding $H^2(0,\ell)\hookrightarrow L^4(0,\ell)$. Letting $c_e$ be the embedding constant, 
\begin{flalign}
V(\state)&\le  c_u\normm{\state}^2+\frac{\alpha}{2}\int_0^{\ell} w^4 \, d\xi\notag\\
&\le c_u\normm{\state}^2+\frac{\alpha c_e}{2}\left(\int_0^{\ell} w_{\xi \xi}^2\, d\xi\right)^2 \notag \\
&\le c_u\normm{\state}^2+\frac{\alpha c_e}{2(EI)^2}\normm{\state}^4. \label{eq:upper}
\end{flalign}
Set $c_h=\alpha c_e/{2(EI)^2}$ in the above inequality to complete the proof.
\end{proof}

The derivative of the Lyapunov function along trajectories of the railway track IVP exists for every $\state_0\in D(\A)$ and continuously differentiable input. In the next lemma, if $C_d=0$, set $1/C_d=\infty$.

\begin{lem}\label{lem:derivative}
Let $[0,T]$ be the interval of existence of the classical solution $\state(t)$ to the railway track IVP with $\state_0\in D(\A)$ and $\inp \in C^1([0,T];\cs)$. Then, the Lyapunov function $V(\state(t))$ is differentiable with respect to time. Moreover, let $c$ satisfy 
\begin{equation}
0<c< \min\left\{\frac{4\rho aEI}{C_d},\frac{4\rho a k\mu}{\mu^2+4\rho a k },\frac{4\rho a k\mu}{1+4\rho a k}\right\}.
\end{equation}
Then, there are positive constants $\epsilon_3$ and $\omega$ such that the derivative $\dot{V}(\state(t))$ satisfies for all $t\in [0,T]$
\begin{equation}
\dot{V}(\state(t))\le \epsilon_3\norm{\inp (t)}{\cs}^2-\omega V(\state(t)). \label{expdecay}
\end{equation}
\end{lem}

\begin{proof}
Since $\state_0\in D(\mc{A})$ and $\inp \in C^1([0,T];\cs)$, the state $\state(t)$ is differentiable on $[0,T]$.
The derivative of the Lyapunov function along trajectories of the railway track IVP is
\begin{flalign}
\dot{V}(\state(t))=2 \int_0^{\ell} &EI w_{\xi \xi} \dot{w}_{\xi \xi}+k w\dot{w}+ \alpha w^3 \dot{w}+\rho a \dot{v}v\notag\\
&+c\dot{w}v+cw\dot{v}\, d\xi.
\end{flalign}
Substituting the time derivatives from the railway track model (\ref{eq:rail-IVP}) leads to
\begin{flalign}
\dot{V}&(\state(t))=2 \int_0^{\ell} EI w_{\xi \xi} v_{\xi \xi}+k wv+ \alpha w^3 v\notag\\
&-v((EIw+C_dv)_{\xi \xi \xi \xi}+kw+\mu v+\alpha w^3+u(\xi,t)) \notag \\
&+cv^2-\frac{c}{\rho a} w((EIw+C_dv)_{\xi \xi \xi \xi}+kw+\mu v\notag\\
&+\alpha w^3+u(\xi,t)) \, d\xi.
\end{flalign}
Performing repeated integration by parts and using the boundary conditions lead to  
\begin{flalign}
&\dot{V}(\state(t))=-2\left[\left(v+\frac{c}{\rho a}w\right)(EIw+C_dv)_{\xi \xi \xi}\right.\notag\\
&\left.-\left(v_{\xi}+\frac{c}{\rho a}w_{\xi}\right)(EIw+C_dv)_{\xi \xi}\right]_0^{\ell}\notag\\
&-2\int_0^{\ell} \frac{EIc}{\rho a}(w_{\xi \xi})^2+\frac{kc}{\rho a}w^2+\frac{\alpha c}{\rho a}w^4+(\mu-c)v^2\notag\\
&\allowdisplaybreaks+C_dv_{\xi \xi}^2\,d\xi-2\int_0^{\ell}\frac{\mu c}{\rho a}wv+\frac{C_d c}{\rho a} w_{\xi \xi}v_{\xi \xi}\, d\xi\notag\\
&-2\int_0^{\ell}u(\xi,t)(v+\frac{c}{\rho a}w) \, d\xi.
\end{flalign}
Young's inequalities (such as inequality (\ref{eq:young})) are used to bound  the product terms.
 Letting $\epsilon_1$, $\epsilon_2$ and $\epsilon_3$ be positive constants,
\begin{flalign}
\dot{V}(\state(t))\le \epsilon_3\norm{\inp (t)}{\cs}^2&-\frac{2}{\rho a}\int_0^{\ell}(EI-\frac{C_d\epsilon_2}{2})c(w_{\xi \xi})^2\notag\\
&+(k-\frac{\mu\epsilon_1}{2}-\frac{c}{2\rho a\epsilon_3})c\,w^2+\alpha c \,w^4\notag\\
&+(\rho a\mu-\rho ac-\frac{\mu c}{2\epsilon_1}-\frac{\rho a \epsilon_3}{2})v^2\notag\\
&+(\rho a-\frac{c}{2\epsilon_2})C_d(v_{\xi \xi})^2\, d\xi \label{eq: expdecay}.
\end{flalign}
Define  the constant 
\begin{flalign}
 \omega_0=\frac{2c}{\rho a}\min& \left\lbrace 1-\frac{C_d \epsilon_2}{2EI},1-\frac{\mu\epsilon_1}{2k}-\frac{c}{2\rho a \epsilon_3 k},\right.\notag\\
&\quad \left. \frac{\mu}{c}-1-\frac{\mu}{2\epsilon_1\rho a}-\frac{\epsilon_3}{2c},\frac{\rho a}{c}-\frac{1}{2\epsilon_2} \right\rbrace . \label{constC}
\end{flalign}
The constant $\omega_0$ needs to be positive. Thus, the constants $\epsilon_1$, $\epsilon_2$, and $\epsilon_3$ are required to satisfy 
\begin{flalign}
\allowdisplaybreaks 0&<\epsilon_1<\frac{2k}{\mu}-\frac{c}{\rho a \epsilon_3 \mu},\label{eq:e1}\\
\allowdisplaybreaks \frac{c}{2\rho a}&<\epsilon_2<\frac{2EI}{C_d},\label{eq:e2}\\
\allowdisplaybreaks 0&<\epsilon_3<2\mu-2c-\frac{\mu c}{\epsilon_1\rho a}.\label{eq:e3}
\end{flalign} 
There is a positive number $\epsilon_2$ satisfying (\ref{eq:e2}) if
\begin{equation}
0<c< \frac{4\rho aEI}{C_d}.
\end{equation}
Also, inequalities (\ref{eq:e1}) and (\ref{eq:e3}) have a solution for $\epsilon_1$ and $\epsilon_3$ if 
\begin{flalign}
0<c<\min\left\{\frac{4\rho a k \mu}{\mu^2+4\rho a k },\frac{4\rho a k \mu}{1+4\rho a k}\right\}.
\end{flalign}
Inequality (\ref{eq: expdecay}) can then be re-written as
\begin{flalign}
&\dot{V}(\state(t))\le \epsilon_3\norm{\inp (t)}{\cs}^2 \label{Lypine1}\\
&-\omega_0 \int_0^{\ell} EI (w_{\xi \xi})^2+kw^2+\alpha w^4+\rho a v^2+C_d(v_{\xi \xi})^2\, d\xi.\notag
\end{flalign}
Using Young's inequality, an upper bound on the Lyapunov function is
\begin{equation}
	V(\state(t))\le r\int_0^{\ell} EI (w_{\xi \xi})^2+kw^2+\alpha w^4+\rho a v^2\, d\xi, \label{Lypine2}
\end{equation}
where
\begin{equation}
r=1+\max\{\frac{c\epsilon_1}{k},\frac{c}{\rho a\epsilon_2}\}.	
\end{equation}
Use this upper bound in (\ref{Lypine1})
\begin{equation}
	\dot{V}(\state(t))\le \epsilon_3\norm{\inp (t)}{\cs}^2-\frac{\omega_0}{r} V(\state(t)).
\end{equation}
Set $\omega =\omega_0/r$ to complete the proof.
\end{proof}

The next theorem uses the Lyapunov function to show that a unique classical solution exists on arbitrary intervals of time for a large class of inputs. It also ensures exponential stability of the solution for some inputs.

\begin{thm}
\label{thm-classical}
Let $c_l$, $c_u$, $c_h$, $\omega$, and $\epsilon_3$ be the same constants as in Lemma \ref{lem:bound} and Lemma \ref{lem:derivative}. If $\state_0 \in D(\mc{A})$ and $\inp \in C^1(\mathbb{R}^+;\cs)$, then the unique classical solution $\state(t)$ to the railway track IVP exists for all $t\ge0$ and satisfies
\begin{flalign}\label{ineq3}
\normm{\state(t)}^2\le e^{-\omega t}&\left(\frac{c_u}{c_l}\normm{\state_0}^2+\frac{c_h}{c_l}\normm{\state_0}^4\right)\notag\\
&+\frac{\epsilon_3}{\omega}(1-e^{-\omega t})\max_{s\in[0,t]}\norm{\inp (s)}{\cs}^2.
\end{flalign}
Moreover, if there are positive constants $u_0$ and $\delta$ so that $\norm{\inp( t)}{\cs}\le u_0e^{-\delta t}$, then $\| \state(t) \|$ exponentially decays to zero. \label{thm:lyapunov}
\end{thm}
\begin{proof}
For every $\bar{T}>0$, consider the input $\inp$ over the bounded interval $[0,\bar{T}]$ and define $R\coloneqq\norm{\inp}{{L^2(0,\bar{T};\cs)}}$. According to Theorem \ref{thm:existence mild} and \ref{thm:existence classical}, for every $\state_0\in D(\A)$ and $\inp \in C^1([0,\bar{T}];\cs)$, with $\norm{\inp}{{L^2(0,\bar{T};\cs)}}\le R$, there is an interval $[0,T]$, $T=T(\state_0,R)\le \bar{T}$, over which a classical solution to the railway track model (\ref{eq:rail-IVP}) exists.

Now use Lemma \ref{lem:derivative}, and apply Gr\"onwall's lemma \cite[Thm.~1.4.1]{zettl2005} to inequality (\ref{expdecay}) to obtain
\begin{equation}
V(\state(t))\le e^{-\omega t}V(\state_0)+\epsilon_3\int_0^t e^{-\omega(t-s)}\norm{\inp (s)}{\cs}^2ds, \label{eq:lyapunov upper bound}
\end{equation}
for all $t\in[0,T]$.
This yields
\begin{equation}
V(\state(t))\le e^{-\omega t}V(\state_0)+\frac{\epsilon_3}{\omega}(1-e^{-\omega t})\max_{s\in[0,t]}\norm{\inp (s)}{\cs}^2. \label{eq:lyapunov upper bound2}
\end{equation}
From  Lemma  \ref{lem:bound}, it follows that for all $t\in[0,T]$, 
\begin{flalign}\label{ineq4}
\allowdisplaybreaks\normm{\state(t)}^2\le e^{-\omega t}&\left(\frac{c_u}{c_l}\normm{\state_0}^2+\frac{c_h}{c_l}\normm{\state_0}^4\right)\notag\\
\allowdisplaybreaks&+\frac{\epsilon_3}{\omega}(1-e^{-\omega t})\max_{t\in[0,\bar T ]}\norm{\inp (t)}{\cs}^2.
\end{flalign}
This classical solution is of course also a mild solution on $[0,T].$ Note that $\| \state (t) \| \leq M$ where $M$ is independent of $T .$ Using \cite[Thm. 6.1.4]{pazy} the solution $\state(t)$ can be extended to $[0, \bar T]. $ Since $\bar T$ was arbitrary, the mild solution exists for all $t>0.$ Since $\state_0 \in D(\A) $ and $\inp \in C^1(\mathbb{R}^+;\cs)$, Theorem \ref{thm:existence classical} then implies that this mild solution is also a classical solution.

Furthermore, if there are positive constants $u_0$ and $\delta$ so that $\norm{\inp(t)}{\cs}\le u_0e^{-\delta t}$, inequality (\ref{eq:lyapunov upper bound}) yields
\begin{equation}
V(\state(t))\le e^{-\omega t} V(\state_0)+\epsilon_3 u^2_0
\left\{ \begin{array}{ll} \frac{e^{-2\delta t}-e^{-\omega t}}{{\omega-2\delta}} &  \omega \neq 2 \delta \\[1ex]
e^{-\omega t}  & \omega = 2 \delta \end{array} \right. .
\notag
\end{equation}
This shows that the Lyapunov function exponentially decays to zero. Since the Lyapunov function $V(\state)$  bounds  the norm of the state by Lemma \ref{lem:bound}, the state will also exponentially decay to zero.
\end{proof}
For inputs with $\sup_{t \geq 0}\norm{\inp (t)}{\cs}< \infty$, Lemma \ref{lem:bound} and Lemma \ref{lem:derivative} result in
\begin{equation}
\dot{V}(\state(t))\le -\omega c_l \normm{\state(t)}^2+\epsilon_3\sup_{t \geq 0}\norm{\inp (t)}{\cs}^2.
\end{equation}
From Definition \ref{def:ISS Lyapunov}, this inequality shows that the Lyapunov function is an ISS Lyapunov function on $D=D(\A)$.

\section{STABILITY of MILD SOLUTIONS}
If the initial condition is not in $D(\A)$ or the input is not continuously differentiable, there may be a unique mild solution to (\ref{eq: IVP}) even though a classical solution may not exist. In this case, the Lyapunov function from Theorem \ref{thm-classical}   may not be differentiable. Thus, exponential decay cannot be shown through manipulating the derivative of Lyapunov function. However, the proof of Theorem \ref{thm:lyapunov} can be modified to yield a result that ensures the existence and stability of global mild solutions for initial conditions in $\ss$ and inputs in $L^2_{loc}(0,\infty;\cs)$.
\begin{thm}\label{thm:mild}
Let $c_l$, $c_u$, $c_h$, $\omega$, and $\epsilon_3$ be the same constants as in Lemma \ref{lem:bound} and Lemma \ref{lem:derivative}. If  $\state_0\in \ss$ and $\inp \in L^2_{loc}(0,\infty;\cs)$, then the unique mild solution, $\state(t)$, to the railway track IVP exists globally. For every $t>0$, the mild solution satisfies
\begin{flalign}\label{ineq2}
\norm{\state}{{C(0,t;\ss)}}^2+\omega\norm{\state}{{L^2(0,t;\ss)}}^2\le&\frac{c_u}{c_l}\normm{\state_0}^2+\frac{c_h}{c_l}\normm{\state_0}^4\notag\\
&+\frac{\epsilon_3}{c_l}\norm{\inp}{{L^2(0,t;\cs)}}^2. 
\end{flalign}
\end{thm}

\begin{proof}
For every $\bar{T}>0$, consider the input $\inp$ over the bounded interval $[0,\bar{T}]$ and define $R\coloneqq\norm{\inp}{{L^2(0,\bar{T};\cs)}}$. According to Theorem \ref{thm:existence mild}, for every $\state_0\in \ss$ and $R$, there is an interval $[0,T]$, $T\le \bar{T}$, over which a unique mild solution $\state(t)$ to the railway track IVP exists. Pick a sequence $\inp_n\in C^{1}([0,T];\cs)$, with $\norm{\inp_n}{{L^2(0,T;\cs)}}\le R$, for all $n\in \mathbb{N}$, that converges to $\inp$ in $L^2(0,T;\cs)$. Also, pick a sequence of initial conditions $\state_0^n\in D(\A)$ that converges to $\state_0$ in $\ss$. Such sequences always exist since $C^1([0,T],\cs)$ is dense in $L^2([0,T];\cs)$, and $D(\A)$ is densely embedded in $\ss$. Corresponding to each initial condition $\state_0^n$ and input $\inp_n(t)$ is a unique classical solution $\state_n(t)$, $t\in[0,T]$, ensured by Theorem \ref{thm:existence classical}. This sequence of solutions also satisfies the equation (\ref{eq:mild solution}) of mild solutions. The mild solution (\ref{eq:mild solution}) is continuous with respect to $\state_0\in \ss$ and $\inp \in L^2(0,T;\cs)$. See \cite[Thm.~6.1.2]{pazy} for Lipschitz continuity with respect to initial conditions, and \cite[Proposition~5.2]{edalatzadehSICON} for Lipschitz continuity with respect to inputs. It follows that
\begin{gather}\label{eq:converge}
\state_n\to \state \text{ in } C([0,T];\ss)\\ 
\text{ as } \state_0^n\to \state_0 \text{ in } \ss \text{ and } \inp_n\to \inp \text{ in } L^2(0,T;\cs).\notag
\end{gather}

The Lyapunov function is differentiable for every pair $(\state_n(t),\inp_n(t))$, and from Lemma \ref{lem:derivative}, its derivative satisfies 
\begin{equation}
\dot{V}(\state_n(t))\le \epsilon_3\norm{\inp_n(t)}{\cs}^2-\omega V(\state_n(t)),
\end{equation}
for all $t\in [0,T]$. Taking the integral yields
\begin{flalign}\label{eq:limit}
V(\state_n(t))\le & V(\state_0^n)+\epsilon_3\int_0^t\norm{\inp_n(s)}{\cs}^2\, ds\notag\\
&-\omega\int_0^t V(\state_n(s))\, ds.
\end{flalign}
From Lemma \ref{lem:bound}, the Lyapunov function satisfies
\begin{flalign}
|V(\state_2)-V(\state_1)|\le& c_u\left|\normm{\state_2}^2-\normm{\state_1}^2\right|\notag\\
& \quad+c_h\left|\normm{\state_2}^4-\normm{\state_1}^4\right|\notag,
\end{flalign}
for all $\state_1$ and $\state_2$ in $\ss$. After some manipulation, it follows that
\begin{flalign}
&|V(\state_2)-V(\state_1)|\le\\
& \quad \left(c_u+ c_h(\normm{\state_2}^2+\normm{\state_1}^2)\right)\left(\normm{\state_2}+\normm{\state_1}\right)\normm{\state_2-\state_1}.\notag
\end{flalign}
Thus, for every $s\in[0,T]$, the sequence $V(\state_n(s))$ converges to $V(\state(s))$. The convergence is also uniform; that is, define
\begin{equation}
r=\sup_n\norm{\state_n}{{C([0,T];\ss)}},
\end{equation}
find that
\begin{flalign}
&|V(\state_n(s))-V(\state(s))|\le \max_{s\in[0,T]} |V(\state_n(s))-V(\state(s))|\notag\\
&\le \left(2c_ur+ 4c_hr^3\right)\norm{\state_n-\state}{{C([0,T];\ss)}}.\notag
\end{flalign}
Thus, by the uniform convergence theorem, the integral in (\ref{eq:limit}) converges. These together with (\ref{eq:converge}) imply that
\begin{equation}\notag
V(\state(t))+\omega\int_0^t V(\state(s))\, ds\le V(\state_0)+\epsilon_3\int_0^t\norm{\inp (s)}{\cs}^2\, ds
\end{equation}
for all $t\in [0,T]$. Apply Lemma \ref{lem:bound} to this inequality and take the maximum of both side over $[0,t]$ to obtain
\begin{flalign}
c_l\norm{\state}{{C(0,t;\ss)}}^2+\omega c_l\norm{\state}{{L^2(0,t;\ss)}}^2&\le c_u\normm{\state_0}^2+c_h\normm{\state_0}^4\notag\\
&+\epsilon_3\norm{\inp}{{L^2(0,t;\cs)}}^2. 
\end{flalign}
This inequality shows that the mild solution can be extended to the interval $[0,\bar{T}]$ \cite[Thm.~6.1.4]{pazy}. Since $\bar{T}$ was  arbitrary, the mild solution exists globally.
\end{proof}
For inputs in $PC(\R^+;\cs)$, a similar density argument can be applied to prove the ISS of railway track IVP.
\begin{cor}
The railway track IVP is input-to-state stable (ISS) in the sense of Definition \ref{def:ISS stability}.
\end{cor}
\begin{proof}
For every $T>0$, $PC([0,T];\cs) \subset L^2_{loc}(0,\infty;\cs)$; thus, Theorem \ref{thm:mild} ensures that a unique mild solution $\state(t)$ exists for all inputs $\inp$ in $PC([0,T];\cs)$. Consider a sequence of initial conditions $\state_0^n\in D(\A)$ converging to $\state_0$ in $\ss$, and also a sequence of inputs $\inp_n\in C^1([0,T];\cs)$ converging uniformly to $\inp\in PC([0,T];\cs)$. Let $\state_n(t)$ be the classical solution to railway track IVP with initial condition $\state_0^n$ and input $\inp_n(t)$. This solution also satisfies (\ref{eq:mild solution}) which ensures that
\begin{gather}
\state_n\to \state \text{ in } C([0,T];\ss)\\ 
\text{ as } \state_0^n\to \state_0 \text{ in } \ss \text{ and } \inp_n(t)\to \inp (t) \text{ uniformly}.\notag
\end{gather}
See \cite[Thm.~6.1.2]{pazy} for Lipschitz continuity with respect to initial conditions, and \cite[Proposition~5.2]{edalatzadehSICON} for Lipschitz continuity with respect to inputs. Use Theorem \ref{thm:lyapunov} to obtain
\begin{flalign}\label{last}
\normm{\state_n(t)}^2\le e^{-\omega t}&\left(\frac{c_u}{c_l}\normm{\state^n_0}^2+\frac{c_h}{c_l}\normm{\state^n_0}^4\right)\notag\\
&+\frac{\epsilon_3}{\omega}\max_{t\in[0,T]}\norm{\inp_n (t)}{\cs}^2.
\end{flalign}
This inequality continuously depends on the norm of initial conditions and inputs. Taking the limit yields a similar inequality for $\state(t)$ with $\state_0$ and $\inp(t)$ replaced. Knowing that $\sup_{t \geq 0}\norm{\inp(t)}{\cs}< \infty$, the ISS property in Definition \ref{def:ISS stability} follows immediately.
\end{proof}
\section{Conclusions and Future Research}
The stability and well-posedness of a nonlinear railway track model was established in this paper. Using a suitable Lyapunov function, it was proved that the model admits a global (in time) classical solution for a continuously differentiable input. The solution is also  exponentially stable. For less regular inputs, belonging only to $L_{loc}^2 (0,\infty; \mathbb U)$ or $PC(\R^+;\cs)$,  existence and stability of a mild solution as well as input-to-state stability (ISS) of the model were established.
Current research is concerned with extending the results of this paper to more general nonlinear structural models.

\bibliographystyle{IEEEtran}
\bibliography{library}
\end{document}